\documentclass[final,12pt]{elsarticle}
\usepackage{srcltx}
\usepackage{eurosym}
\usepackage{mathtools}
\usepackage{amsmath}
\usepackage{cases}
\usepackage{amsfonts}
\usepackage{amssymb}
\usepackage{amsthm}
\usepackage{pgfplots}

\usepackage{graphicx}
\usepackage{mathrsfs}
\usepackage{xcolor}
\usepackage{xpatch}
\usepackage{exscale}
\usepackage{latexsym}
\usepackage{tikz}
\usepackage{caption}
\usepackage{environ}

\usepackage{cancel,ulem}
\xpatchcmd{\MaketitleBox}{\hrule}{}{}{}
\xpatchcmd{\MaketitleBox}{\hrule}{}{}{}
\usepackage[colorlinks,plainpages=true,pdfpagelabels,hypertexnames=true,colorlinks=true,pdfstartview=FitV,linkcolor=blue,citecolor=red,urlcolor=black]{hyperref}
\PassOptionsToPackage{unicode}{hyperref}
\PassOptionsToPackage{naturalnames}{hyperref}
\usepackage{enumerate}
\usepackage[shortlabels]{enumitem}
\usepackage{bookmark}
\usepackage{wasysym}
\usepackage{esint}
\usepackage[ddmmyyyy]{datetime}
\usepackage[margin=2cm]{geometry}
\makeatletter
\g@addto@macro\normalsize{%
	\setlength\abovedisplayskip{4pt}
	\setlength\belowdisplayskip{4pt}
	\setlength\abovedisplayshortskip{4pt}
	\setlength\belowdisplayshortskip{4pt}
}
\numberwithin{equation}{section}
\everymath{\displaystyle}
\usepackage[capitalize,nameinlink]{cleveref}
\crefname{section}{Section}{Sections}
\crefname{subsection}{Subsection}{Subsections}
\crefname{condition}{Condition}{Conditions}
\crefname{hypothesis}{Hypothesis}{Conditions}
\crefname{assumption}{Assumption}{Assumptions}
\crefname{lemma}{Lemma}{Lemmas}
\crefname{definition}{Definition}{Definitions}
\crefname{figure}{figure}{figures}
\crefname{question}{Question}{Questions}

\crefformat{equation}{\textup{#2(#1)#3}}
\crefrangeformat{equation}{\textup{#3(#1)#4--#5(#2)#6}}
\crefmultiformat{equation}{\textup{#2(#1)#3}}{ and \textup{#2(#1)#3}}
{, \textup{#2(#1)#3}}{, and \textup{#2(#1)#3}}
\crefrangemultiformat{equation}{\textup{#3(#1)#4--#5(#2)#6}}%
{ and \textup{#3(#1)#4--#5(#2)#6}}{, \textup{#3(#1)#4--#5(#2)#6}}%
{, and \textup{#3(#1)#4--#5(#2)#6}}

\Crefformat{equation}{#2Equation~\textup{(#1)}#3}
\Crefrangeformat{equation}{Equations~\textup{#3(#1)#4--#5(#2)#6}}
\Crefmultiformat{equation}{Equations~\textup{#2(#1)#3}}{ and \textup{#2(#1)#3}}
{, \textup{#2(#1)#3}}{, and \textup{#2(#1)#3}}
\Crefrangemultiformat{equation}{Equations~\textup{#3(#1)#4--#5(#2)#6}}%
{ and \textup{#3(#1)#4--#5(#2)#6}}{, \textup{#3(#1)#4--#5(#2)#6}}%
{, and \textup{#3(#1)#4--#5(#2)#6}}

\crefdefaultlabelformat{#2\textup{#1}#3}
\numberwithin{equation}{section}

\newtheorem{theorem} {Theorem}[section]
\newtheorem{proposition} {Proposition}[section]
\newtheorem{lemma}{Lemma}[section]
\newtheorem{corollary}{Corollary}[section]

\newtheorem{counter-example}{counter-example}[section]
\newtheorem{remark} {Remark}[section]
\newtheorem{definition} {Definition}[section]

\def\N{\mathbb{N}}
\def\CC{{\rm \kern.24em \vrule width.02em height1.4ex depth-.05ex \kern-.26emC}}

\def\TagOnRight

\def\AA{{it I} \hskip-3pt{\tt A}}

\def\QQ{\rlap {\raise 0.4ex \hbox{$\scriptscriptstyle |$}} {\hskip -0.1em Q}}

\makeatletter
\newcommand{\vo}{\vec{o}\@ifnextchar{^}{\,}{}}
\makeatother

\def\YYint#1#2#3{{\setbox0=\hbox{$#1{#2#3}{\iint}$}
		\vcenter{\hbox{$#2#3$}}\kern-.50\wd0}}

\def\XXint#1#2#3{{\setbox0=\hbox{$#1{#2#3}{\int}$}
		\vcenter{\hbox{$#2#3$}}\kern-.50\wd0}}

\makeatletter
\def\namedlabel#1#2{\begingroup
	\def\@currentlabel{#2}%
	\label{#1}\endgroup
}
\makeatother
\makeatletter
\newcommand{\rmh}[1]{\mathpalette{\raisem@th{#1}}}
\newcommand{\raisem@th}[3]{\hspace*{-1pt}\raisebox{#1}{$#2#3$}}
\makeatother

\newcommand{\descitem}[2]{\item[#1] \label{#2}}
\newcommand{\descref}[2]{\hyperref[#1]{\textnormal{\textcolor{black}{}\textcolor{blue}{\bf #2}\textcolor{black}{}}}}

\newcommand{\dref}[2]{\hyperref[#1]{\textcolor{black}{(}\textcolor{blue}{\bf #2}\textcolor{black}{)}}}
\newcommand{\be} {\begin{eqnarray}}
	\newcommand{\ee} {\end{eqnarray}}
\newcommand{\Bea} {\begin{eqnarray*}}
	\newcommand{\Eea} {\end{eqnarray*}}
\newcommand{\rr}{\rightarrow}

\newcommand{\de} {\delta}

\newcommand{\p}  {\prime}

\newcommand{\la} {\lambda}
\newcommand{\si} {\sigma}

\newcommand{\f}{\infty}

\newcommand{\R}{\mathbb{R}}

\newcommand{\al}{\alpha}
\newcommand{\ga}{\gamma}

\newcommand{\tcb}{\textcolor{blue}}

\newcommand{\norm}[1]{\left|\hspace{-0.2mm}\left| #1 \right|\hspace{-0.2mm}\right|}
\newcommand{\abs}[1]{\left| #1\right|}

\newcommand{\RN}[1]{%
	\textup{\uppercase\expandafter{\romannumeral#1}}%
}

\newcounter{whitney}
\refstepcounter{whitney}

\newcounter{ineqcounter}
\refstepcounter{ineqcounter}
\makeatletter
\def\ps@pprintTitle{%
	\let\@oddhead\@empty
	\let\@evenhead\@empty
	\def\@oddfoot{}%
	\let\@evenfoot\@oddfoot}
\makeatother
\usepackage[doublespacing]{setspace}
\usepackage[titletoc,toc,page]{appendix}

\usepackage{pgf,tikz}
\usetikzlibrary{arrows}
\usetikzlibrary{decorations.pathreplacing}
\allowdisplaybreaks
\usepackage{cancel,soul,ulem}

\usepackage{etoolbox}
\usepackage{hyperref}

\usepackage{lipsum}

\makeatletter
\def\@mkboth#1#2{}
\newlength\appendixwidth
\preto\appendix{\addtocontents{toc}{\protect\patchl@section}}
\newcommand{\patchl@section}{%
	\settowidth{\appendixwidth}{\textbf{Appendix }}%
	\addtolength{\appendixwidth}{1.5em}%
	\patchcmd{\l@section}{1.5em}{\appendixwidth}{}{\ddt}%
}
\makeatother

\begin{document}
	\begin{frontmatter}
		\title{Higher regularity for  entropy solutions of conservation laws with geometrically constrained discontinuous flux}
		
		\author[myaddress]{Shyam Sundar Ghoshal}
		\ead{ghoshal@tifrbng.res.in}

		\author[myaddress1]{St\'ephane Junca}
		\ead{stephane.junca@univ-cotedazur.fr}

		\author[myaddress]{Akash Parmar}
		\ead{akash@tifrbng.res.in}
		\address[myaddress]{Centre for Applicable Mathematics, Tata Institute of Fundamental Research, Post Bag No 6503, Sharadanagar, Bangalore - 560065, India.}
		
		\address[myaddress1]{Universit\'e C\^ote d'Azur,  LJAD,  Inria \& CNRS,
			Parc Valrose,
			06108 Nice, 
			France.}

		
		\begin{abstract}
            For the Burgers equation, the entropy solution  becomes  instantly $BV$ with only $L^\infty$ initial data.  
            For conservation laws with genuinely nonlinear  discontinuous flux, it is well known that the $BV$ regularity of entropy solutions is lost.  Recently, this regularity has been proved to be fractional with $s= 1/2$. Moreover, for less nonlinear flux the solution has still a fractional regularity $ 0<s  \leq 1/2$.  The resulting general rule is  the regularity of entropy solutions for a discontinuous flux is less than for a smooth flux.   
            In this paper, an optimal  geometric condition on the discontinuous flux is used  to recover the same regularity as for the smooth flux with the same kind of non-linearity.    
		\end{abstract}

		\begin{keyword}
			Conservation laws\sep Interface \sep Discontinuous flux\sep Cauchy problem\sep Regularity \sep $BV$ functions  \sep Fractional $BV$ spaces.
			\MSC[2020] 35B65\sep35L65  \sep 35F25 \sep 35L67  \sep 26A45 \sep35B44.
		\end{keyword}
		\end{frontmatter}
	\tableofcontents
	\section{Introduction}
Scalar conservation laws with discontinuous flux arise in traffic flow modeling of vehicular traffic on highways and use to predict congestion and control them (\cite{LWR, traffic}). Another example is the continuous sedimentation of solid particles in the liquid (\cite{diehl,diehl4}). It occurs in the model of two-phase flow in the porous media which is significant to many scientific and industrial explorations such as petroleum engineering and many others \cite{petro}. With this brief discussion of the applications of scalar conservation laws, we can continue to discuss some literature. For more details, one can see \cite{burger2, burger4, ion}.

The well-posedness of the scalar conservation laws with discontinuous flux has been studied extensively over the last few decades \cite{Kyoto, boris2, BGG, STG, P09}. The existence of the solution given by the several numerical schemes \cite{AJG, boris1, burger1, SAT, G_JHDE, tower}. However, the regularity of the entropy solution has not been studied. In general, 
for conservation laws with discontinuous flux, the entropy solutions are not regular, not in $BV$ even if the initial data is in $BV$ \cite{AS}.  The discontinuity of the flux usually forbids the decay of the total variation. Even for constant initial data, the entropy solution is not constant. 

However, a  fractional regularity occurs \cite{fractional}. 
The fractional regularity discovered has been proven to be optimal in many cases. 
The maximal regularity belongs to some fractional $BV$ spaces $BV^s$.   For the general 
genuinely nonlinear case it is  at most  $BV^{1/2}$.

Fractional $BV$ spaces  have been  introduced \cite{junca1} to precisely quantify the regularity of entropy solutions with shocks that do not belong to $BV$ \cite{CJ1, Cheng83}. Such regularity is optimal  \cite{SAJJ}, for smooth nonlinear flux  as for the Burgers equation and less nonlinear flux. In  the recent work \cite{fractional} the fractional $BV$ spaces are  useful to apprehend with precision  the regularity of entropy solutions with  discontinuous  convex flux.   

In this paper, the unexpected regularity of entropy solutions has been given for a special  family of discontinuous flux.
This family has been highlighted in \cite{S}. It provides solutions as regular as continuous flux.   If the initial data belongs to $BV$  the entropy solutions keep this $BV$ regularity. Moreover, with the strongest non-linearity and $L^\infty$ initial data, the solutions become $BV$ as for the Burgers equation. In this article, the same family is considered with more general non-linearity. In this case, the $BV^s$ framework  is needed. Propagation of $BV^s$ regularity or smoothing effect in these spaces is similar to continuous flux although the fractional total variation is not decreasing.

We proceed to describe the precise form of the discontinuous flux in detail. The scalar conservation law analyzed in this paper focuses on a discontinuous flux that consists of one interface. The specific form of the equation is as follows:
	\begin{eqnarray}\label{1.1}
		\left\{\begin{array}{rlll}
			u_{t}+f(u)_{x}&=0, &\mbox{ if }& x>0,  t>0,\\
			u_{t}+g(u)_{x}&=0,&\mbox{ if }& x<0,  t>0,\\
			u(x,0)&=u_{0}(x),  &\mbox{ if }& x\in\R,
		\end{array}\right.
	\end{eqnarray}
	where $u:\R\times[0,\f)\rr\R$ is unknown, $u_0(x)\in L^{\f}(\R)$ is the initial data and $f$, $g$ are the fluxes. 
 The family of fluxes studied satisfies the  following assumptions:
 \begin{description}
    \descitem{A1.}{a1} Fluxes $f$ and $g$ belongs to $C^{1}(\R,\R)$  and are strictly convex. 
    \descitem{A2.}{a2} Compatible fluxes: $\min f= \min g$,  i.e., $f(\theta_{f})=g(\theta_{g})$, where $\theta_{f}$, $\theta_{g}$ are the critical points of $f, g$ respectively.
    \descitem{A3.}{a3} Non-degeneracy condition, there exist two numbers $p\ge1, q\ge1$ such that, for any compact set $K$, there exist positive numbers $C_{1}, C_{2}$ for all $u\not=v$, and $u, v\in K$,
    \begin{eqnarray}\label{fluxc}
		\frac{|f'(u)-f'(v)|}{|u-v|^{p}}>C_{1} >0  &\mbox{ and }&\frac{|g'(u)-g'(v)|}{|u-v|^{q}}>C_{2}>0 .
	\end{eqnarray} 
 \end{description}

 \begin{center}
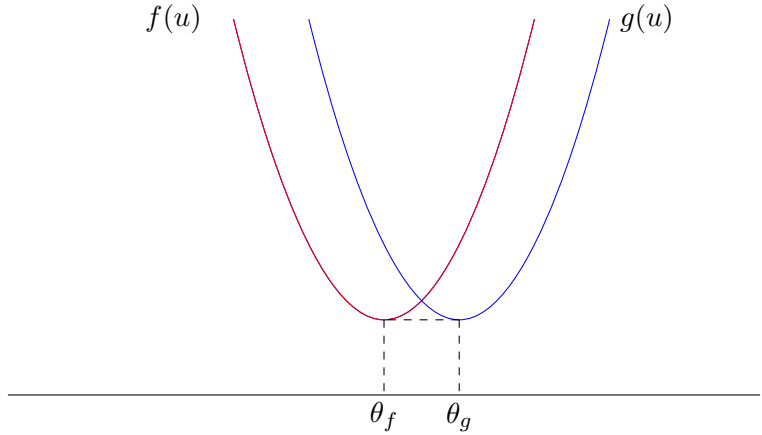

     \begin{tikzpicture}
         \draw (-5,-1)--(5,-1);
         \draw[blue] (0,0) parabola (2,4);
         \draw[blue] (0,0) parabola (-2,4);
         \draw[red] (0,0) parabola (2,4);
         \draw[red] (0,0) parabola (-2,4);
         \draw[blue] (1,0) parabola (3,4);
         \draw[blue] (1,0) parabola (-1,4);
         \draw[dashed] (1,0)--(0,0);
         \draw[dashed] (1,0)--(1,-1);
         \draw[dashed] (0,0)--(0,-1);
         \draw (-2.8,4) node{\small $f(u)$};
         \draw (3.5,4) node{\small $g(u)$};
         \draw (0,-1.29) node{\small $\theta_{f}$};
         \draw (1,-1.29) node{\small $\theta_{g}$};
     \end{tikzpicture}
     \captionof{figure}{An illustration of two compatible fluxes $f$ and $g$.}
 \end{center}
 The assumption \descref{a2}{(A2)} is fundamental,  else loss of regularity immediately occurs near the interface. 
Examples with less regularity have been built in \cite{fractional}. That means that geometric assumption  \descref{a2}{(A2)}  is optimal to keep the same regularity as for the case with smooth flux. 

The nonlinear assumption \descref{a3}{(A3)} implies strict convexity and is the necessary condition to have a smoothing effect with at least a  fractional derivative  for only $L^\infty$ data \cite{CJ1, CJJ,  SAJJ, JaX}. Without this assumption but still, with a nonlinear convex flux, the regularity has to be quantified in bigger generalized $BV$ spaces \cite{CJLO, GJC}.
\medskip 

It is well-known about scalar conservation laws that it does not admit the classical solution for all time $t$. Thus one needs the following notion of a weak solution. 
 \begin{definition}[{\bf Weak solution}]\label{weak}  
		A function $u\in C(0,T;L^{1}_{loc}(\R))$ is said to be a weak solution of the problem \eqref{1.1}  if
		\begin{equation*}
			\int\limits_{0}^{\f}\int\limits_{\R} u\frac{\partial \phi}{\partial t}+F(x,u)\frac{\partial \phi}{\partial x}\mathrm{d}x \ \mathrm{d}t +\int\limits_{\R}^{} u_{0}(x)\phi(x,0)\mathrm{d}x =0,
		\end{equation*}
		for all $\phi\in D(\R\times\R^+)$, where  the flux  $F(x,u)$ is given as $F(x, u)=H(x)f(u)+(1-H(x))g(u)$, and $H(x)$ is Heaviside function.
  \end{definition}
  By using the definition of the weak solution one can derive the condition called  the Rankine-Hugoniot condition, at $x=0$, $u$ satisfies the following for almost all $t$,
		\begin{equation}\label{RH}
			f(u^{+}(t))=g(u^{-}(t)),
		\end{equation} 
where $u^{+}(t)=\lim\limits_{x\rr0+}u(x,t)$ and $u^{-}(t)=\lim\limits_{x\rr0-}u(x,t)$.
  The existence of traces like a $BV$ function is an important usual feature for weak solutions of conservation laws \cite{P09}. The left and right traces $u^{-}$ and $u^{+}$ are crucial for the wellposedness as well as the regularity of the weak solution. The existence of the interface traces has been established in \cite{Kyoto} by using the Hamilton-Jacobi type equation. 

  Due to the fact that the weak solutions are not unique one needs to take into account the ``entropy condition" to get the uniqueness of the weak solution. For $f=g$, Kru\v{z}kov \cite{Kruzkov} gave a generalized entropy condition. However, due to the discontinuity of flux at the interface the Kru\v{z}kov entropy condition is insufficient for each side of the interface. Therefore, an additional "interface entropy condition" is required, which was introduced in \cite{Kyoto}.
  \begin{definition}[{\bf Entropy solution} \cite{Kyoto}]
		A weak solution $u$  of the problem \eqref{1.1} is said to be an entropy solution     if   it satisfies Kruzkov entropy solutions on each side of the interface $x=0$ and if  the    following ``interface entropy condition"  is fulfilled for almost all $t>0$, 
		\begin{equation} \label{iec}
			|\{t: f'(u^{+}(t))>0> g'(u^{-}(t))\}|=0,
		\end{equation}
	\end{definition}
Let us give a summary of this paper. The main results are given and commented on in the next Section \ref{main}. Then, important tools are recalled in Section \ref{def} before proving the main results. The proof of the main results is given in Section \ref{mainproof} along with many useful lemmas. 

\section{Main Results}\label{main}
As we have mentioned  the fluxes are convex and precisely satisfy the assumptions \descref{a1}{(A1)}, \descref{a2}{(A2)} and \descref{a3}{(A3)}. Let us denote the $g_{-}^{-1}$, $f_{+}^{-1}$ denote the inverse of $g, f$ for domain where $g'(u)\le0$ and $f'(u)\ge 0$ respectively.  Notice that the existence of a minimum for $f$ and $g$ are always assumed in this paper as it allows the critical behavior of the entropy solution. If we are assuming that $f$ and $g$  have no minimum it implies that the fluxes are strictly increasing or decreasing which is handled in \cite{AS}. In this case, when the minimum of the two fluxes coincide, $f(\theta_{f})=g(\theta_{g})$, it is similar to the case $f=g$ where there is no interface, except that the proof and the estimates are more complicated and need the theory of conservation laws with  discontinuous flux. In particular, if $f$ and $g$ are  uniformly convex then a $BV$ smoothing effect occurs. The fractional $BV$ spaces $BV^s, 0 < s \le 1$ are recalled in the next section. For $s=1$, $BV^1$ is $BV$.  
\begin{theorem}[{\bf Smoothing effect for compatible fluxes with $L^\infty$ initial data}]\label{linf}
		Let $f$ and $g$ be convex fluxes with the same minimum and satisfy the non-degeneracy condition \eqref{fluxc}, i.e., condition \descref{a3}{(A3)} along with assumptions \descref{a1}{(A1)} and \descref{a2}{(A2)}. Let $u(\cdot, t)$ be  the entropy solution of \eqref{1.1} corresponding to an initial data $u_{0}\in L^{\f}(\R)$. Then $u(t,\cdot)\in BV^s_{loc}(\R)$ for $t>0$ where $s=\min\{1/p,1/q\}$. 
	\end{theorem}
	\begin{remark}
		Note that the exponent $s=\min\{1/p,1/q\}$ in Theorem \ref{linf} is optimal due to the counter-examples in \cite{finer,CJ1,SAJJ}. In particular, for uniform convex and compatible fluxes $f$ and $g$,  there is a smoothing effect in $BV$ as for the case $f=g$ \cite{lax1,ol}.
	\end{remark}
	Once we know about the regularity of entropy solution in Theorem \ref{linf}, it is natural to see whether $TV^s(u(t,\cdot))$ satisfies an explicit estimate in terms of time variable $t$. In the next proposition, we provide a sufficient condition to obtain a $1/t^\ga$ estimate for an appropriate choice of $\ga$.
		\begin{proposition}[{\bf Explicit $BV^s$ estimates}]\label{prop:estimate}
			Let $f$ and $g$ be two $C^2$ fluxes such that $f(\theta_{f})=g(\theta_{g})$ and satisfying the non-degeneracy condition \eqref{fluxc} with exponents $p,q$ respectively. Let $u(\cdot, t)$ be  the entropy solution of \eqref{1.1} corresponding to an initial data $u_{0}\in L^{\f}(\R)$. We assume that there exist $\de,r>0$ such that 
			\begin{equation}\label{condition:estimate}
				\operatorname*{ess-sup}\limits_{(-\de,0)}u_0>\theta_g+r\mbox{ and  }\operatorname*{ess-sup}\limits_{(0,\de)}u_0<\theta_f-r,
			\end{equation}
			then we have for all $M>0$, 
			\begin{equation}\label{estimate-precise}
				TV^{s}(u(\cdot,t), [-M, M])\le \frac{C(f,g,M)}{\min\{t^{\frac{1}{qs}},t^{\frac{1}{ps}}\}}\left[\max\{1,t/\de,1/r\}\right]^{\frac{pqs+1}{s^2}}+2(2||u_{0}||_{\f})^{1/s}
				\mbox{ for }t>0,
			\end{equation}
			where $s=\min\{1/p,1/q\}$ and $C$ is depending only on $f,g,M$.
		\end{proposition}
  
 Another nice property of $BV^s$ solutions is known for all Lipschitz flux,   
 $BV^s$ regularity  in space of the initial data is propagated for all time \cite{junca1, JH}. It means the corresponding entropy solution $u(x,t)$ has the $BV^s$ regularity in space for all time $t>0$. For discontinuous flux, the situation is more complicated due to the interface. To use the Lax-Oleinik formula with an interface  \cite{Kyoto} we limit ourselves to convex $C^1$ fluxes but the non-degeneracy condition \descref{a3}{(A3)} is not required. 
\begin{theorem}[{\bf Propagation of the regularity  for compatible fluxes}]\label{bvs1}
	Let $f$ and $g$ be the $C^{1}$ convex fluxes such that $\min f=\min g$, i.e. they satisfy the assumptions \descref{a1}{(A1)} and \descref{a2}{(A2)}. Let $u(\cdot, t)$ be  the entropy solution of \eqref{1.1} corresponding to an initial data $u_{0}\in BV^{s}(\R)$ for $s\in(0, 1)$. Then $u(\cdot,t)\in BV^s(\R)$ for $t>0$.
	\end{theorem}
 For $s=1$, the propagation of the $BV$ regularity already proved in \cite{S} is recovered for less nonlinear fluxes.
 \medskip

 Now, the two previous theorems can be mixed to get another result. 
	As a consequence of the Theorems  \ref{linf} and \ref{bvs1},  if $u_0$ belongs to $BV^s$ then two cases occur. If $s$ is too small then $u_0$ is regularized. If $s$ is too big then the initial regularity is propagated. Precisely, the following corollary is stated.
    To use  Theorem  \ref{linf}, the flux has to satisfy the non-degeneracy condition \descref{a3}{(A3)}  to get the $BV^s$ smoothing effect.  
	\begin{corollary}[{\bf Optimal regularity for  $BV^s$ initial data and compatible fluxes}]\label{bvs}
		Let $f$ and $g$ be the fluxes such that $f(\theta_{f})=g(\theta_{g})$ and  the fluxes satisfy the non-degeneracy condition \eqref{fluxc}. Let $u(\cdot, t)$ be  the entropy solution of \eqref{1.1} corresponding to an initial data $u_{0}\in BV^{s}(\R)$ for $s\in(0, 1)$. Then, $u(\cdot,t)\in BV^{s_1}([-M,M])$ for $t>0,M>0$ where $s_{1}=\min\{\max\{p^{-1}, s\}, \max\{q^{-1}, s\}\}$.
		%
	\end{corollary}

		\section{Preliminaries on scalar conservation laws with an interface}\label{def}
		The fundamental paper used here is \cite{Kyoto} where Adimurthi and Gowda settle an important  foundation of the theory on scalar conservation laws with an interface and two convex fluxes. 
		In this paper, the author proposed the natural entropy condition which means that no information comes only from the interface but crosses or goes toward the interface.   Such entropy condition is in the spirit of Lax-entropy conditions for shock waves. To make this paper self-contained, we recall some definitions and results. The following theorem can be found in \cite{Kyoto} Lemma 4.9 on page 51. It is  a Lax-Oleinik or Lax-Hopf formula for the initial value problem \eqref{1.1}.
	\begin{theorem}[\cite{Kyoto}]\label{constant}
		Let $u_{0}\in L^{\f}(\R)$, then there exists the entropy solution $u(\cdot, t)$ of \eqref{1.1} corresponding to an initial data $u_{0}$. Furthermore, there exist Lipschitz curves $R_{1}(t)\geq  R_{2}(t)\ge0$ and $L_{1}(t)\leq  L_{2}(t)\le0$, monotone functions $z_{\pm}(x,t)$ non-decreasing in $x$ and non-increasing in $t$ and $t_{\pm}(x,t)$ non-increasing in $x$ and non-decreasing in $t$ such that the solution $u(x,t)$ can be given by the explicit formula for  almost all $t>0$,
		\begin{eqnarray*}
			u(x,t)=\left\{\begin{array}{lllll}
				(f')^{-1}\left(\frac{x-z_+(x,t)}{t}\right)& \mbox{ if }& x\ge R_{1}(t),\\
				(f')^{-1}\left(\frac{x}{t-t_{+}(x,t)}\right) &\mbox{ if }& 0\le x<R_{1}(t),\\
				(g')^{-1}\left(\frac{x-z_-(x,t)}{t}\right) &\mbox{ if }& x\le L_{1}(t),\\
				(g')^{-1}\left(\frac{x}{t-t_{-}(x,t)}\right) &\mbox{ if }& L_{1}(t)<x<0.
			\end{array}\right.
		\end{eqnarray*}
		Furthermore, if $f(\theta_{f})\ge g(\theta_{g})$ then $R_1(t)=R_2(t)$ and if $f(\theta_{f})\le g(\theta_{g})$ then $L_1(t)=L_2(t)$. We also  have only three cases and following formula to compute the solution: 
		
		\begin{description}
			\descitem{Case 1:}{thm1-case-1} $L_{1}(t)=0$ and $R_{1}(t)=0$,
			\begin{eqnarray*}
				u(x, t)=\left\{\begin{array}{lllll}
					u_{0}(z_{+}(x, t))&\mbox{ if }& x>0,\\
					u_{0}(z_{-}(x, t))& \mbox{ if }& x<0.
				\end{array}\right.
			\end{eqnarray*}
			\descitem{Case 2:}{thm1-case-2} $L_{1}(t)=0$ and $R_{1}(t)>0$, then
			\begin{eqnarray*}
				u(x,t)=\left\{\begin{array}{lllll}
					f_{+}^{-1}g(u_{0}(z_{+}(x,t)))&\mbox{ if }& 0<x<R_{2}(t),\\
					f_{+}^{-1}g(\theta_{g}) &\mbox{ if }& R_{2}(t)\le x\le R_{1}(t),\\
					u_{0}(z_{-}(x,t))& \mbox{ if }& x<0.
				\end{array}\right.
			\end{eqnarray*}
			\descitem{Case 3:}{thm1-case-3} $L_{1}(t)<0$, $R_{1}(t)=0$, then
			\begin{eqnarray*}
				u(x,t)=\left\{\begin{array}{lllll}
					g_{-}^{-1}f(u_{0}(z_{-}(x,t)))&\mbox{ if }& L_{2}(t)<x<0,\\
					u_{0}(z_{-}(x, t))&\mbox{ if }&  x\le L_{1}(t),\\
					g_{-}^{-1}f(\theta_{f})& \mbox{ if }& L_{1}(t)<x<L_{2}(t).
				\end{array}\right.
			\end{eqnarray*}
		\end{description}
	\end{theorem}
	\begin{center}
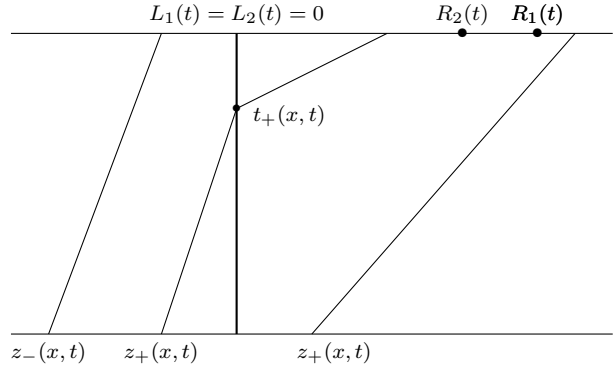
\label{figure-1}
		\begin{tikzpicture}
			\draw (-3,0)--(5,0);
			\draw[thick] (0,0)--(0,4);
			\draw (-3,4)--(5,4);
			\draw (2,4)--(0,3);
			\draw (0,3)--(-1,0);
			\draw (4.5,4)--(1,0);
			\draw (-1,4)--(-2.5,0);
			\draw (.7,2.9) node{\scriptsize $t_{+}(x,t)$};
			\draw (4,4.25) node{\scriptsize $R_{1}(t)$};
			\draw (0,3) node{\tiny $\bullet$};
			\draw (4,4) node{\scriptsize $\bullet$};
			\draw (3,4.25) node{\scriptsize $R_{2}(t)$};
			\draw (3,4) node{\scriptsize $\bullet$};
			\draw (4,4.25) node{\scriptsize $R_{1}(t)$};
			\draw (0,4.25) node{\scriptsize $L_{1}(t)=L_{2}(t)=0$};
			\draw (-1,-.25) node{\scriptsize $z_{+}(x,t)$};
			\draw (-2.5,-.25) node{\scriptsize $z_{-}(x,t)$};
			\draw (1.3,-.25) node{\scriptsize $z_{+}(x,t)$};
		\end{tikzpicture}
		\captionof{figure}{An illustration of solution for \descref{thm1-case-2}{Case 2} and $L_{i}(t)$ and $R_{i}(t)$ curves}
	\end{center}
	There is a maximum principle  for such entropy solutions, but more complicate than for $f=g$,
	\begin{equation}\label{eq:max}
		\|u\|_\infty \leq \max \left(  \|u_0\|_\infty ,   \sup_{|v| \leq \|u_0\|_\infty }|f^{-1}_+ ( g(v))|,  \sup_{|v| \leq \|u_0\|_\infty }|g^{-1}_- ( f(v))|  \right)=:S_{f,g,\norm{u_0}_\f}.
	\end{equation}
	In the above theorem, the curves $R_i$ play an important role to understand the structure inherited by the entropy solution. They separate the regions of three types of `characteristics' as introduced in \cite{Kyoto}.  
\bigskip

Now, the definition of $BV^s(I)$ is recalled for $I$ an interval of $\R$. First, the $s$-fractional total variation is,  for $ 0 < s \le 1$ and $p=1/s$, 
\begin{equation}
TV^s(u,I)=\sup_{\sigma } \sum_{i=1}^n |u(x_i) - u(x_{i-1})|^p,
\end{equation}
where the supremum is taken over all subdivisions  $\sigma$ of the interval $I$. Such a subdivision is a finite ordered subset of $I$,   $ \sigma = \{x_i,\; i=0,\ldots n\}$ for a $n \in \N$  and   $x_0 < x_1 < \ldots < x_n$. 

Second, $BV^s(I)$ is the space of real functions on I  with bounded $s$-fractional total variation, called also $p$-variation.
The exponent $s$ means that $BV^s$ is almost but different than the Sobolev spaces $W^{s,p}(I)$  (when $I$ is bounded) \cite{junca1}.  Thus the exponent $s$ corresponds to  the fractional derivative $s$.   

Some applications of fractional $BV$  spaces for scalar conservation laws with convex flux can be found in the non-exhaustive list 
\cite{CJJ, CJ1, CJLO, SAJJ}.  $BV^s$ is compactly embedded in $L^1_{loc}$ and capture the shock structure as $BV$ but with less regularity.
\bigskip 

Before we prove the main results in the next section, we need to state some important elementary lemmas. In these lemmas, we prove the non-Lipschitz and Lipschitz regularity of the singular map.  The two first lemmas are in \cite{fractional}, thus the proofs are omitted.
	\begin{lemma}[ {\bf H\"older continuity of the inverse} \cite{fractional}]\label{lemma:Holder}
		Let $g\in C^1(\R)$ be satisfying the non-degeneracy \eqref{fluxc} with exponent $q$. Then $(g^{\p})^{-1}$ is H\"older continuous with exponent $1/q$.
	\end{lemma}
	\begin{lemma}[{\bf Flatness of the antiderivative} \cite{fractional}]\label{lemma:g-condition}
		Let $g$ be a $C^2$ function satisfying \eqref{fluxc} with exponent $q$ then $g_+$ satisfies \eqref{fluxc} with exponent $q+1$ on domain $(\theta_g,\f)$.
	\end{lemma}
The following lemma establishes the Lipschitz continuity of the ``singular maps" away from $\theta_g$ and $\theta_f$. Before stating the lemma, let us define the singular maps. Let  $g_{-}^{-1}$ and $f_{+}^{-1}$ are the inverse of $g$ and $f$ respectively, with appropriate domains given by:

	\begin{eqnarray}\label{inverses}
 \left\{\begin{array}{rlllll}
     g_{-}^{-1}:((g')^{-1} ( - \infty),   (g')^{-1} (0)] \rightarrow \R  \\
     f_{+}^{-1}: [(f')^{-1} (0), (f')^{-1} (+ \infty))  \rightarrow \R.
        \end{array}\right.
	\end{eqnarray}
   The functions $f^{-1}_{+}g(\cdot)$ and $g^{-1}_{-}f(\cdot)$, namely, singular maps, plays crucial role in traferring information via the interface from left-to-right and right-to-left respectively.
\begin{lemma}[{\bf Regularity of singular maps}]\label{lipschitz}
		Let $f$ and $g$ be the  $C^{1}(\R)$ and convex functions with $f(\theta_{f})=g(\theta_{g})$ with restricted flatness, i.e., satify \descref{a1}{(A1)}, \descref{a2}{(A2)} and \descref{a3}{(A3)}. Let $K$ is any compact set of $\R$. Then $f^{-1}_{+}g(\cdot)$ and $g^{-1}_{-}f(\cdot)$ are H\"older continuous functions.
  Moreover, $f^{-1}_{+}g(\cdot)$ is a Lipschitz continuous function on $K\setminus[\theta_{g}-\varepsilon,\theta_{g}+\varepsilon]$ and $g^{-1}_{-}f(\cdot)$ is a Lipschitz continuous function on $K\setminus[\theta_{f}-\varepsilon,\theta_{f}+\varepsilon]$ for any fixed $\varepsilon>0$.
	\end{lemma}
 Indeed,  if $x$ is far from $\theta_g$, then $g(x)$ is far from $\theta_f$  since $f(\theta_f)=g(\theta_g)$. Thus, $f^{-1}_{+}g(\cdot)$ is Lipschitz far from $\theta_g$. The symmetric result occurs for  $g^{-1}_{-}f(\cdot)$  far from $\theta_f$. Now, the detailed proof is given. 
\begin{proof}
    Let $x_{1}, x_{2}\in K$ then there exist $y_{1}, y_{2}\ge\theta_{f}$ such that $g(x_{1})=f(y_{1})$ and $g(x_{2})=f(y_{2})$.  Without loss of generality we can assume that $g(x_{1})\not=g(x_{2})$ because if $g(x_{1})=g(x_{2})$ then result holds anyway. As $f^{-1}_+$ is increasing, we get $y_{1}, y_{2}\ge\theta_{f}$. Let us consider the following 
		\begin{eqnarray*}
			\frac{|f^{-1}_{+}g(x_{1})-f^{-1}_{+}g(x_{2})|^{p+1}}{|x_{1}-x_{2}|}&=&\frac{|f^{-1}_{+}g(x_{1})-f^{-1}_{+}g(x_{2})|^{p+1}}{|g(x_{1})-g(x_{2})|}\cdot\frac{|g(x_{1})-g(x_{2})|}{|x_{1}-x_{2}|}\\
			&=&\frac{|f^{-1}_{+}f(y_{1})-f^{-1}_{+}f(y_{2})|^{p+1}}{|f(y_{1})-f(y_{2})|}\cdot\frac{|g(x_{1})-g(x_{2})|}{|x_{1}-x_{2}|}\\
			&=&\frac{|y_{1}-y_{2}|^{p+1}}{|f(y_{1})-f(y_{2})|}\cdot\frac{|g(x_{1})-g(x_{2})|}{|x_{1}-x_{2}|}
		\end{eqnarray*} 
  From Lemma \ref{lemma:g-condition} we have
  \begin{eqnarray*}
      |f(x)-f(y)|\ge\frac{C_{2}}{p+1}|x-y|^{p+1} &\mbox{ for any }x\not=y,
  \end{eqnarray*}
 and $g$ is Lipschitz continuous function, we have $\abs{g(x_{1})-g(x_{2})}\leq c_1\abs{x_{1}-x_{2}}$ where $c_1$ depends on $g$ and $K$. Therefore,  we get a constant $C$, 
		\begin{equation*}
			\frac{|f^{-1}_{+}g(x_{1})-f^{-1}_{+}g(x_{2})|^{p+1}}{|x_{1}-x_{2}|}\le C.
		\end{equation*}
  Similarly, we can prove that
            \begin{equation*}
			\frac{|g^{-1}_{-}f(x_{1})-g^{-1}_{-}f(x_{2})|^{q+1}}{|x_{1}-x_{2}|}\le C.
		\end{equation*}
  Thus the singular maps are H\"older continuous functions. 
 \\
For any $x_{1}, x_{2}\in K \backslash [\theta_g -\varepsilon, \theta_g + \varepsilon]$ such that $g(x_{1}), g(x_{2})\ge g(\theta_{g})+\varepsilon_{1}$ for a positive $\varepsilon_1$ only depending on $\varepsilon$. Then for some $\varepsilon_{0}>0$ only depending on $\varepsilon_1$ there exist $y_{1}, y_{2}\ge \theta_{f}+\varepsilon_{0}$ such that $g(x_{1})=f(y_{1})$ and $g(x_{2})=f(y_{2})$. Let us consider the following 
		\begin{eqnarray*}
			\frac{|f^{-1}_{+}g(x_{1})-f^{-1}_{+}g(x_{2})|}{|x_{1}-x_{2}|}&=&\frac{|f^{-1}_{+}g(x_{1})-f^{-1}_{+}g(x_{2})|}{|g(x_{1})-g(x_{2})|}\cdot\frac{|g(x_{1})-g(x_{2})|}{|x_{1}-x_{2}|}\\
			&=&\frac{|f^{-1}_{+}f(y_{1})-f^{-1}_{+}f(y_{2})|}{|f(y_{1})-f(y_{2})|}\cdot\frac{|g(x_{1})-g(x_{2})|}{|x_{1}-x_{2}|}\\
			&=&\frac{|y_{1}-y_{2}|}{|f(y_{1})-f(y_{2})|}\cdot\frac{|g(x_{1})-g(x_{2})|}{|x_{1}-x_{2}|}\\
                &=&\frac{1}{f'(\xi)}\cdot\frac{|g(x_{1})-g(x_{2})|}{|x_{1}-x_{2}|}
		\end{eqnarray*}  
 where $\xi\in(y_{1}, y_{2})$. Now from the Lipschitz continuous of $g$ and from the convexity of $f$ we get
 \begin{equation*}
			\frac{|f^{-1}_{+}g(x_{1})-f^{-1}_{+}g(x_{2})|}{|x_{1}-x_{2}|}\le C.
		\end{equation*}
 In a similar fashion, we can prove
  \begin{equation*}
			\frac{|g^{-1}_{-}f(x)-g^{-1}_{-}f(y)|}{|x_{1}-x_{2}|}\le C.
		\end{equation*}

  \end{proof}

	\section{Proof of main results}\label{mainproof}
In this section, we first prove the smoothing effect in fractional $BV$ space with only $L^\infty$ initial data, 
that is Theorem  \ref{linf}. The proof is detailed in the long subsection \ref{th2} with an explicit $BV^s$ estimate, Proposition \ref{prop:estimate}. After, similar arguments are used in the last short subsection to prove the propagation of fractional $BV$ regularity and the smoothing effect when the initial data is already in $BV^s$. 
	\subsection{Smoothing effect for compatible fluxes with $L^\infty$ initial data}\label{th2}
   The following long proof contains the main ingredients to study the regularity of the entropy solution through the interface. 
   After, an explicit fractional $BV$  estimate is given in a favorable case. 
	\begin{proof}[Proof of Theorem \ref{linf}]
		We assume that $f(\theta_{f})=g(\theta_{g})$, hence, from Theorem \ref{constant} we have $L_2(t)=L_1(t)$ and $R_2(t)=R_1(t)$, hence, it is enough to consider the following two cases. The second case is similar so, only the first case is detailed.
		\begin{description}
			
			\descitem{Case (i):}{case-1} If $L_{1}(t)=0$ and $R_{1}(t)\geq 0$.\\
			Consider the partition $\sigma=\{-M=x_{-n}\le\cdots<x_{-1}<x_{0}\le0<x_{1}<\cdots<x_{m}\le R_{1}(t)<x_{m+1}<\cdots\le x_{n}=M \}$ and let $s=\min\{1/p, 1/q\}$, 
			\begin{eqnarray*}
				\sum_{i=-n}^{n}|u(x_{i},t)-u(x_{i+1},t)|^{\frac{1}{s}}&=&\sum_{i=-n}^{-1}|u(x_{i},t)-u(x_{i+1},t)|^{\frac{1}{s}}+|u(x_{0},t)-u(x_{1},t)|^{\frac{1}{s}}\\
				&+&\sum_{i=1}^{m-1}|u(x_{i},t)-u(x_{i+1},t)|^{1/s}+|u(x_{m},t)-u(x_{m+1},t)|^{\frac{1}{s}}\\
				&+&\sum_{i=m+1}^{n-1}|u(x_{i},t)-u(x_{i+1},t)|^{\frac{1}{s}}.\\
			\end{eqnarray*}
			By using the explicit formula  from Theorem \ref{constant} we get,
			\begin{eqnarray}
				\sum_{i=-n}^{n}|u(x_{i},t)-u(x_{i+1},t)|^{1/s}&\le&\underbrace{\sum_{i=-n}^{-1}|u(x_{i},t)-u(x_{i+1},t)|^{\frac{1}{s}}}_\text{I}\nonumber\\
				&+&\underbrace{\sum_{i=1}^{m-1}|f^{-1}_{+}g(u_{0}(z_{+}(x_{i},t)))-f^{-1}_{+}g(u_{0}(z_{+}(x_{i+1},t)))|^{\frac{1}{s}}}_\text{II}\nonumber\\
				&+&\underbrace{\sum_{i=m+1}^{n-1}|u(x_{i},t)-u(x_{i+1},t)|^{\frac{1}{s}}}_\text{III}+ 2(2||u_{0}||_{\f})^{\frac{1}{s}}.\label{eq:terms}
			\end{eqnarray}
			Now we wish to estimate the terms \RN{1}, \RN{2}, and \RN{3}.  The simplest terms  \RN{1},  \RN{3} are estimated as in \cite{junca1}.  First taking the \RN{1} and \RN{3} into the account, since $f,g$ are satisfying the flux non-degeneracy condition \eqref{fluxc}, by Lemma \ref{lemma:Holder}, the maps $u\mapsto (g')^{-1}(u)$ and $u\mapsto (f')^{-1}(u)$ are H\"older continuous with exponents $1/q$ and $1/p$.
			From the Theorem \ref{constant},  for $x<0$, the solution of \eqref{1.1} is given by $u(x, t)=(g')^{-1}\left(\frac{x-z_{-}(x, t)}{t}\right)$. Then for $-M\leq x_i<x_{i+1}\leq 0$ we have
			\begin{eqnarray*}
				|u(x_{i},t)-u(x_{i+1},t)|^{q}&=&\left|(g')^{-1}\left(\frac{x_{i}-z_{-}(x_{i},t)}{t}\right)-(g')^{-1}\bigg(\frac{x_{i+1}-z_{-}(x_{i+1},t)}{t}\bigg) \right|^{q}\\
				&\le&\bigg(C_{2}^{-q^{-1}}\bigg|\frac{x_{i}-z_{-}(x_{i},t)}{t}-\frac{x_{i+1}-z_{-}(x_{i+1},t)}{t}\bigg|^{q^{-1}} \bigg)^{q},
			\end{eqnarray*}
			using triangle inequality we obtain,
			\begin{equation*}
				|u(x_{i},t)-u(x_{i+1},t)|^{q}\le C_{2}^{-1}\bigg|\frac{x_{i}-x_{i+1}}{t}\bigg|+C_{2}^{-1}\bigg|\frac{z_{-}(x_{i},t)-z_{-}(x_{i+1},t)}{t}\bigg|.
			\end{equation*}
			Since $|x_i|, |x_{i+1}|\le M$ and $x=z_{-}(x, t)+g'(u(x,0))t$ hence, we get the $TV^{q^{-1}}u(\sigma\cap[-M,0])$,
			%
			\begin{equation}\label{calc-1}
				TV^{q^{-1}}u(\sigma\cap[-M,0])\le \frac{4M}{C_{2} t}+\frac{1}{C_2}\sup\left\{\abs{g^{\p}(v)};\,\abs{v}\leq \norm{u_0}_{L^{\f}(\R)}\right\}.
			\end{equation}
			In similar fashion, for the term $\RN{3}$ we have, 
			\begin{equation}\label{estimation-III}
				TV^{p^{-1}}u(\sigma\cap[R_1(t),M])\le \frac{4M}{C_{1} t}+\frac{1}{C_1}\sup\left\{\abs{f^{\p}(v)};\,\abs{v}\leq \norm{u_0}_{L^{\f}(\R)}\right\}.
			\end{equation}
			Next we focus on computing term II. Without loss of generality we can choose a point $x'\in(0,R_{2}(t))$ to take  $x'= R_{2}(t)$ such that $z_{+}(x',t)<0$, suppose there is no such point $x'$, then $R_{2}(t)=0$, since $R_1(t)=R_2(t)=0$ which implies that the term $\RN{2}$ does not appear. 
			Let characteristics emanating from $z_{+}(x', t)$ be hitting the interface at $t_{+}(x', t)$ then by the monotonicity of $t_{+}(\cdot, t)$ we get for all $x\in(0,x')$,	
			\begin{equation}\label{t+}
				0< t_{+}(x', t)\le t_{+}(x, t)< t ,
			\end{equation} 
			also by the monotonicity of $z_{+}(\cdot, t)$ we get for all $x\in(0,x')$,
			\begin{equation}\label{zt}
				z_{+}(0+, t)\le z_{+}(x, t)\le z_{+}(x', t)<0,
			\end{equation}
			which implies that 
			\begin{equation}\label{z0}
				\frac{z_{+}(0+, t)}{t_{+}(x, t)}\le \frac{z_{+}(x, t)}{t_{+}(x, t)}\le \frac{z_{+}(x', t)}{t_{+}(x, t)}<0.
			\end{equation}
			Hence, from \eqref{t+} and \eqref{zt} we get
			\begin{eqnarray}\label{zt+}
				\frac{z_{+}(0+, t)}{t_{+}(x, t)}&\ge& \frac{z_{+}(0+, t)}{t_{+}(x', t)},\\\label{zt++}
				\frac{z_{+}(x', t)}{t}&\ge& \frac{z_{+}(x', t)}{t_{+}(x, t)}.
			\end{eqnarray}
			Therefore, from \eqref{z0}, \eqref{zt+} and \eqref{zt++} one can conclude that for $x\in(0,x')$,
			\begin{equation}\label{bounded}
				\frac{z_{+}(0+,t)}{t_{+}(x',t)}\le \frac{z_{+}(x,t)}{t_{+}(x,t)}\le \frac{z_{+}(x',t)}{t}<0.
			\end{equation}
			Since $(g')^{-1}$ is an increasing function and $(g')^{-1}(0)=\theta_{g}$, the above inequality \eqref{bounded} implies that, 
			\begin{equation}\label{lb:g-inv}
				(g')^{-1}\left(-\frac{z_{+}(x,t)}{t_{+}(x,t)}\right)\geq \de_0:=(g')^{-1}\left(-\frac{z_{+}(x',t)}{t}\right)>\theta_g\mbox{ for }x\in(0,x').
			\end{equation}
			Then, the map $a\mapsto g((g^{\p})^{-1}(a)$ is Lipschitz continuous for $a\geq g^{\p}(\de_0)=:c_0>0$. 
			As $f(\theta_{f})=g(\theta_{g})$ we have if $g((g^{\p})^{-1}(a))$ is 
			away from $g(\theta_{g})$ then also away from $f(\theta_{f})$ which implies that,
			\begin{equation}\label{func:fgg}
				a\mapsto f_+^{-1}g\left((g')^{-1}(a)\right)
			\end{equation} 
			is a H\"older continuous  function with exponent $1/q$ for $a\geq c_0>0$. 
			\begin{center}
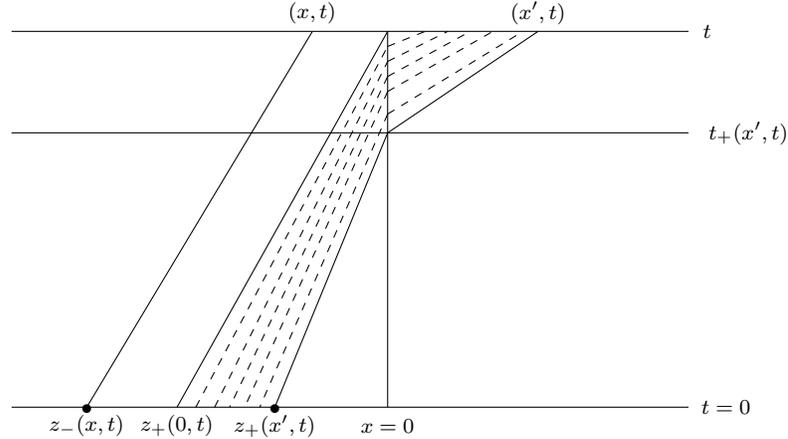

				\begin{tikzpicture}
					\draw (-5,0)--(4,0);
					\draw (0,0)--(0,5);
					\draw(-5,5)--(4,5);
					\draw (-5,3.65)--(4,3.65);
					\draw (2,5)--(0,3.65);
					\draw (0,3.65)--(-1.5,0);
					\draw [dashed](0,3.9)--(-1.7,0);\draw [dashed](0,3.9)--(1.8,5);
					\draw [dashed](0,4.2)--(-1.9,0);
					\draw [dashed](0,4.4)--(-2.1,0);\draw [dashed](0,4.2)--(1.4,5);
					\draw [dashed](-2.3,0)--(0, 4.6);\draw [dashed](0,4.4)--(1.1,5);
					\draw [dashed](-2.55,0)--(0, 4.8);\draw [dashed](0,4.6)--(.8,5);
					\draw (0,5)--(-2.8,0);\draw [dashed](0,4.8)--(.5,5);
					\draw(-1,5)--(-4,0);
					\draw (0,-.25)node{\scriptsize $x=0$};
					\draw (-1.5,-.02) node{\scriptsize $\bullet$};
					\draw (-2.8,-.25) node{\scriptsize$z_{+}(0,t)$};
					\draw (-1.5,-.25) node{\scriptsize$z_{+}(x',t)$};
					\draw (2,5.25) node{\scriptsize$(x',t)$};
					\draw (4.25,5) node{\scriptsize $t$};
					\draw (4.8, 3.65) node{\scriptsize $t_{+}(x',t)$};
					\draw (4.5,0) node{\scriptsize $t=0$};
					\draw (-4,-.02) node{\scriptsize $\bullet$};
					\draw (-4,-.25) node{\scriptsize$z_{-}(x,t)$};
					\draw (-1,5.25)node{\scriptsize$(x,t)$};
				\end{tikzpicture}
    \captionof{figure}{An illustration of solution for the region $(0, x')$ at time $t$.}
			\end{center}
			From Theorem \ref{constant} for all $x\in(0, x')$ we have
			\begin{equation}\label{g*}
				g(u_{0}(z_{+}(x, t)))=g(u(0-, t_{+}(x, t)))=(g')^{-1}\left(-\frac{z_{+}(0, t_{+}(x, t))}{t_{+}(x, t)}\right).
			\end{equation}
			To get a more precise estimate with respect to t, we refine the choice of $x^{\p}$. We wish to choose a point $x^\p\in(0,R_2(t))$ such that $t_+(x^\p+,t)\leq t/2\leq t_+(x^\p-,t)$. Since $t_+(x,t)\rr t$ as $x\rr 0+$, we can always get $x^\p$ satisfying $t_+(x^\p-,t)\geq t/2$. If we get an $x^\p$ such that $t_+(x^\p+,t)\leq t/2\leq t_+(x^\p-,t)$, we define $\sigma_{1}:=\sigma\cap [0, x')$, $\sigma_{2}:=\sigma\cap [x', R_{2}(t))$. In the other case, definition of $\si_1$ remains same and $\si_2$ is set to be empty.  Now from \eqref{g*} we get, \begin{eqnarray*}
				\RN{2}&=&\sum_{x_{i}\in\sigma_{1}}^{}\left| f_{+}^{-1}g\bigg((g')^{-1}\left(-\frac{z_{+}(x_{i},t)}{t_{+}(x_{i},t)}\right)\bigg)- f_{+}^{-1}g\bigg((g')^{-1}\left(-\frac{z_{+}(x_{i+1},t)}{t_{+}(x_{i+1},t)}\right)\bigg)\right|^{\frac{1}{s}}\\
				&+&\sum_{x_{i}\in\sigma_{2}}^{}\left|(f')^{-1}\left(\frac{x_{i}}{t-t_{+}(x_{i}, t)}\right)-(f')^{-1}\left(\frac{x_{i+1}}{t-t_{+}(x_{i+1},t)}\right)\right|^{\frac{1}{s}}\\
&\le&C(f,g,t,u_0)\cdot\sum_{x_{i}\in\sigma_{1}}^{}\bigg|-\frac{z_{+}(x_{i},t)}{t_{+}(x_{i},t)}+\frac{z_{+}(x_{i+1},t)}{t_{+}(x_{i+1},t)}\bigg|^{\frac{1}{sq}}\\
				&+&C_{1}\cdot\sum_{x_{i}\in\sigma_{2}}^{}\bigg|\frac{x_{i}}{t-t_{+}(x_{i},t)}-\frac{x_{i+1}}{t-t_{+}(x_{i+1},t)}\bigg|^{\frac{1}{sp}}.
			\end{eqnarray*}
			Note that $\abs{t-t_+(x,t)}\geq\abs{t-t_+(x'+,t)}\geq t/2$ for $x\in (x^{\p},R_1(t))$. Since  $s\, p<1$, we obtain
			\begin{align*}
				&\abs{\frac{x_i}{t-t_+(x_i,t)}-\frac{x_{i+1}}{t-t_+(x_{i+1},t)}}^{\frac{1}{sp}}\\
				&\leq \frac{2^{\frac{1}{sp}}\abs{x_i-x_{i+1}}^{\frac{1}{sp}}}{\abs{t-t_+(x_i,t)}^{\frac{1}{sp}}}+2^{\frac{1}{sp}}\abs{\frac{x_{i+1}}{t-t_+(x_i,t)}-\frac{x_{i+1}}{t-t_+(x_{i+1},t)}}^{\frac{1}{sp}}\\
				&\leq \frac{2^{\frac{1}{sp}}}{\abs{t-t_+(x^{\p},t)}^{\frac{1}{sp}}}\abs{x_i-x_{i+1}}^{\frac{1}{sp}}+\frac{(2M)^{\frac{1}{sp}}}{\abs{t-t_+(x',t)}^{\frac{2}{sp}}}\abs{t_+(x_i,t)-t_+(x_{i+1},t)}^{\frac{1}{sp}}\\
				&\leq \frac{2^{\frac{2}{sp}}}{t^{\frac{1}{sp}}}\abs{x_i-x_{i+1}}^{\frac{1}{sp}}+\frac{(8M)^{\frac{1}{sp}}}{t^{\frac{2}{sp}}}\abs{t_+(x_i,t)-t_+(x_{i+1},t)}^{\frac{1}{sp}}.
			\end{align*}
			By a similar argument, we have
			\begin{align}
				&\abs{\frac{z_{+}(x_{i},t)}{t_{+}(x_{i},t)}-\frac{z_{+}(x_{i+1},t)}{t_{+}(x_{i+1},t)}}\nonumber\\
				&\leq \frac{1}{t_{+}(x_{i},t)}\abs{z_{+}(x_{i},t)-z_{+}(x_{i+1},t)}+\frac{\abs{z_+(x_{i+1},t)}}{t_{+}(x_{i},t)t_+(x_{i+1},t)}\abs{t_{+}(x_i,t)-t_+(x_{i+1},t)}.\nonumber
			\end{align}
			Due to the choice of $x^\p$, $t_+(x,t)\geq t/2$ for $x\in(0,x^\p)$. Therefore,
			\begin{align}
				&\abs{\frac{z_{+}(x_{i},t)}{t_{+}(x_{i},t)}-\frac{z_{+}(x_{i+1},t)}{t_{+}(x_{i+1},t)}}\nonumber\\
				&\leq \frac{2}{t}\abs{z_{+}(x_{i},t)-z_{+}(x_{i+1},t)}+\frac{4\abs{z_+(x_{i+1},t)}}{t^2}\abs{t_{+}(x_i,t)-t_+(x_{i+1},t)}.\label{estimate:z+}
			\end{align}
			Hence, by the monotonicity of $x\mapsto t_{+}(x, t)$ and $x\mapsto z_{+}(x, t)$ we estimate,
			\begin{align}
				\RN{2}&\leq \frac{2^{\frac{1}{s}}C(f, g,t,u_0)}{t^{\frac{1}{sq}}}\abs{z_{+}(0+,t)-z_{+}(x^\p-,t)}^{1/sq}\nonumber\\
				&+\frac{4^{\frac{1}{sq}}\abs{z_+(0+,t)}^{\frac{1}{sq}}C(f, g,t,u_0)}{t^{\frac{2}{sq}}}\abs{t_{+}(0+,t)-t_+(x^\p,t)}^{\frac{1}{sq}}
				\nonumber\\
				&+\frac{2^{\frac{2}{sp}}C_1}{t^{\frac{1}{sp}}}\abs{x_1-x_{m+1}}^{\frac{1}{sp}}+\frac{(8M)^{\frac{1}{sp}}C_1}{t^{\frac{2}{sp}}}\abs{t_+(x_1,t)-t_+(x_{m+1},t)}^{\frac{1}{sp}}.
			\end{align}
			Note that $\abs{t_{+}(0+,t)-t_+(x^\p,t)}\leq t/2$ and $\abs{t_+(x_1,t)-t_+(x_{m+1},t)}\leq t/2$. Hence we get
			\begin{equation}\label{estimation-II}
				\RN{2}\le C(f, g,t,u_0,M)\frac{1}{t^{\frac{1}{sp}}}.
			\end{equation}
			Thus, combining \RN{1}, \RN{2} and \RN{3} we obtain,
			\begin{eqnarray}
				\sum_{-n}^{n}|u(x_{i},t)-u(x_{i+1},t)|^{\frac{1}{s}}&\le&C(f, g, t,u_0,M)+2(2||u_{0}||_{\f})^{\frac{1}{s}}. 
			\end{eqnarray}
			\descitem{Case (ii):}{case-2} $R_{1}(t)=0$, $L_{1}(t)<0$. This case follows by a similar argument as in \descref{case-1}{Case (i)}.
		\end{description}
		This proves Theorem \ref{linf}.
	\end{proof}
	Next, we prove Proposition \ref{prop:estimate} which concerns the precise estimate of $TV^s(u(\cdot,t))$ when initial data satisfies condition \eqref{condition:estimate}. This condition means when the initial data are far from critical points of fluxes.  In general, without this condition,  the situation is quite intricate and no estimates are proposed. 
		\begin{proof}[Proof of Proposition \ref{prop:estimate}:]	
			We prove an explicit estimate of term II as in \eqref{eq:terms} and note that for terms I and III we can use explicit estimates \eqref{calc-1}, \eqref{estimation-III}. To this end, it is enough to get an explicit estimate of $C(f,g, M,t,u_0)$ in \eqref{estimation-II}. We note that $C(f,g,M,t,u_0)$ comes from estimation of Lipschitz constant of $f^{-1}_+g$ which depends on domain $[a,m]$ of $f_+^{-1}g$ for $\theta_g<a<m<\f$. Observe that as $a\rr \theta_g$ the Lipschitz constant of $f_+^{-1}g$ blows up. We prove that when initial data is satisfying \eqref{condition:estimate}, we can estimate the effective domain of $f_+^{-1}g$ in terms of $\de,r,t$. Since $u_0$ satisfies \eqref{condition:estimate}, we can re-write \eqref{bounded} as 
			\begin{equation}
				\frac{z_{+}(0+,t)}{t_{+}(x',t)}\le \frac{z_{+}(x,t)}{t_{+}(x,t)}\le \frac{z_{+}(x',t)}{t}\leq \max\left\{\frac{-\de}{t},-g^{\p}(\theta_g+r)\right\}.
			\end{equation}
			We observe the following if $\varphi\in C(\R)$ is an increasing function satisfying $\abs{\varphi(a_1)-\varphi(a_2)}\leq C_1\abs{a_1-a_2}^\al$ with $\al\in(0,1]$ then we have $a_1\geq a_2+C_1^{-1/\al}\abs{\varphi(a_1)-\varphi(a_2)}^{1/\al}$ for $a_1\geq a_2$.  Since $g$ is a $C^2$ function we have 
			\begin{equation}
				(g')^{-1}\left(-\frac{z_{+}(x,t)}{t_{+}(x,t)}\right)\geq \de_0:=\theta_g+C_g\min\left\{{\de}/{t},r\right\}\mbox{ for }x\in(0,x').
			\end{equation}
			By Lemma \ref{lemma:g-condition}, $g_+$ satisfies the non-degeneracy condition \eqref{fluxc} with exponent $q+1$, therefore, we obtain for $a\geq \theta_g+C_g\min\left\{{\de}/{t},r\right\}$, 
			\begin{align*}
				g (g^{\p})^{-1}(a)&\geq g(\theta_g)+C_{g,q}\min\left\{{\de^{q+1}}/{t^{q+1}},r^{q+1}\right\}\\
				&=f(\theta_f)+C_{g,q}\min\left\{{\de^{q+1}}/{t^{q+1}},r^{q+1}\right\}.
			\end{align*}
			Since $f\in C^1(\R)$ we get
			\begin{equation*}
				f_+^{-1}(g (g^{\p})^{-1}(a))\geq \theta_f+C_{f,g,q}\min\left\{{\de^{q+1}}/{t^{q+1}},r^{q+1}\right\},
			\end{equation*}
			for $a\geq \theta_g+C_g\min\left\{{\de}/{t},r\right\}$.
			Let $h$ be defined as $h=f_+^{-1}\circ g $. Then, $h^\p(b)=\frac{g^{\p}(b)}{f^{\p}(f_+^{-1}(g(b)))}$. Since $f$ satisfies \eqref{fluxc}, $\abs{h^\p(b)}\leq C_{g,q,m}\la^{-p}$ for $b$ satisfying $\abs{b}\leq m$ and $f_+^{-1}g(b)\geq \theta_f+\la$. 
			Hence, we have for $a,b\geq \min\left\{(g^\p)^{-1}(\de/t),\theta_g+r\right\}$,
			\begin{equation*}
				\abs{f_+^{-1} g(g^{\p})^{-1}(a)-f_+^{-1} g(g^{\p})^{-1}(b)}\leq C_{f,g}\max\left\{{t^{pq+p}}/{\de^{pq+p}},{r^{-pq-p}}\right\}\abs{a-b}^{\frac{1}{q}}.
			\end{equation*}
			As $1/s\geq \max\{p,q\}$ we get $a^{pq+p}\leq a^{pq+\frac{1}{s}}$ for $a\geq1$. Applying the above observations in estimate \eqref{estimation-II} of term II we obtain \eqref{estimate-precise}.
		\end{proof}
	\subsection{Propagation of the regularity  and smoothing for compatible fluxes}
	Now Theorem \ref{bvs1} is proved and then its Corollary  \ref{bvs}.  
 Many arguments given in the previous subsection are used and shortly recalled in the following proof to prove that the initial fractional regularity is propagated as for a smooth flux. 
	\begin{proof}[Proof of Theorem \ref{bvs1}]
		Due to the assumption $f(\theta_{f})=g(\theta_{g})$, it is enough to consider the following two cases: (i) $L_1(t)=0, R_1(t)\geq0$, (ii) $R_1(t)=0, L_1(t)\leq0$. Likewise Theorem \ref{bvs}, we give detailed proof only for case (i).
		\begin{description}
			\descitem{Case (i):}{thm-bvs1-case-1} If $L_{1}(t)=0$, $R_{1}(t)\ge 0$.\\
			Consider the partition $\{\cdots<x_{-1}<x_{0}\le0=L_{2}(t)=L_{1}(t)<x_{1}<\cdots<x_{m}\le R_{2}(t)=R_{1}(t)<x_{m+1}<\cdots \}$ and fix $s\in(0,1)$,
			\begin{eqnarray*}
				\sum_{i=-\f}^{\f}|u(x_{i},t)-u(x_{i+1},t)|^{1/s}&=&\sum_{i=-\f}^{-1}|u(x_{i},t)-u(x_{i+1},t)|^{1/s}+|u(x_{0},t)-u(x_{1},t)|^{1/s}\\&+&\sum_{i=1}^{m-1}|u(x_{i},t)-u(x_{i+1},t)|^{1/s}+|u(x_{m},t)-u(x_{m+1},t)|^{1/s}\\&+&\sum_{i=m+1}^{\f}|u(x_{i},t)-u(x_{i+1},t)|^{1/s}.
			\end{eqnarray*}
			From the Theorem \ref{constant} we get,
			\begin{eqnarray*}
				\sum_{i=-\f}^{\f}|u(x_{i},t)-u(x_{i+1},t)|^{1/s}&\le&\underbrace{\sum_{i=-\f}^{-1}|u_{0}(z_{-}(x_{i},t))-u_{0}(z_{-}(x_{i+1},t))|^{1/s}}_\text{I}+2(2||u_{0}||_{\f})^{1/s}\\
				&+&\underbrace{\sum_{i=1}^{m-1}|f^{-1}_{+}g(u_{0}(z_{+}(x_{i},t)))-f^{-1}_{+}g(u_{0}(z_{+}(x_{i+1},t)))|^{1/s}}_\text{II}\\
				&+&\underbrace{\sum_{i=m+1}^{\f}|u_{0}(z_{+}(x_{i},t))-u_{0}(z_{+}(x_{i+1},t))|^{1/s}}_\text{III}.\\
			\end{eqnarray*}
			Since the initial data $u_{0}\in BV^{s}(\R)$, hence,  \RN{1} and \RN{3} 
			\begin{equation*}
				\RN{1} \tcb{+} \RN{3}\le TV^{s}(u_{0}).
			\end{equation*}
			Hence, from \eqref{z0}, \eqref{zt+} and \eqref{zt++} one can conclude that for $x\in(0, x')$,
			\begin{equation*}
				\frac{z_{+}(0,t)}{t_{+}(x, t)}\le \frac{z_{+}(x, t)}{t_{+}(x, t)}\le \frac{z_{+}(x',t)}{t}.
			\end{equation*}
			$g$ is a convex function which implies that $(g')^{-1}$ is an increasing function and $(g')^{-1}(0)=\theta_{g}$ hence, $(g')^{-1}\Big(-\frac{z_{+}(x,t)}{t_{+}(x,t)}\Big)$, for $x\in(0,x')$, is  always away from $\theta_{g}$. 
			%
			Therefore, similarly as in the proof of Theorem \ref{linf} we get
			\begin{equation*}
				\RN{2}=\sum_{i=1}^{l-1}|f^{-1}_{+}g(u_{0}(z_{+}(x_{i},t)))-f^{-1}_{+}g(u_{0}(z_{+}(x_{i+1},t)))|^{1/s}\le  C(f, g, t,u_0).
			\end{equation*}
			Hence,
			\begin{equation*}
				\sum_{i=-\f}^{\f}|u(x_{i},t)-u(x_{i+1},t)|^{1/s}\le 2TV^{s}(u_{0})+C(f, g, t)+2(2||u_{0}||_{\f})^{1/s}.
			\end{equation*}
			\descitem{Case (ii):}{thm-bvs1-case-2} $R_{1}(t)=0$, $L_{1}(t)<0$.
			
			This case follows by a similar argument as in \descref{thm-bvs1-case-1}{Case (i)}. 
			Hence, the theorem.
		\end{description}
	\end{proof}
  Now the last main result is proven. It gives the regularity of the solution due to the fractional regularity of the initial data and the smoothing effect. Indeed Corollary  \ref{bvs} is a  direct consequence of Theorem \ref{bvs1} and \ref{linf}.
	\begin{proof}[Proof of Corollary \ref{bvs}]
		We assume $f(\theta_{f})=g(\theta_{g})$ and fluxes satisfies the non-degeneracy condition \eqref{fluxc}.
		Hence, Theorem \ref{linf} and Theorem \ref{bvs1} together imply that the solution is in fractional BV space with exponent $s_{1}=\min\{\max\{s, r_{1}\},$$ \max\{s, r_{2}\}\}$.
	\end{proof}
 \begin{remark}
			With appropriate sufficient conditions, we can obtain estimates of type \eqref{estimate-precise} for Theorem \ref{bvs1} and Corollary \ref{bvs} as well. 
	\end{remark}
	\section{Acknowledgement}
    The authors thank the IFCAM project: "Conservation laws: $BV^s$, interface, and control". The second author would like to thank  TIFR-CAM for hospitality.

\end{document}